\documentclass[a4paper,11pt]{amsart}
\usepackage[all]{xy}
\usepackage{amsfonts}
\usepackage{amssymb}
\usepackage[centertags]{amsmath}
\usepackage{graphicx}
\usepackage{subfigure}
\usepackage{amssymb,amsthm,enumerate,epsfig}

\newcommand{\rank}{\operatorname{rank}}

\newcommand{\spec}{\operatorname{Spec}}
\newcommand{\dis}{\displaystyle}

\def\R{{\mathbb R}}
\def\N{{\mathbb N}}
\def\C{{\mathbb C}}

\def\H{{\mathcal H}}
\def\E{{\mathcal E}}
\def\Q{{\mathbb Q}}
\def\Z{{\mathbb Z}}

\def\e{{\mathbf e}}
\def\x{{\mathbf x}}
\def\y{{\mathbf y}}

\def\h{{\mathbf h}}

\theoremstyle{plain}
 \newtheorem{thm}{Theorem}[section]
 
 \newtheorem{corollary}[thm]{Corollary}
 
 \newtheorem{prop}[thm]{Proposition}
 \theoremstyle{definition}
 \newtheorem{defn}[thm]{Definition}
 \theoremstyle{remark}
 \newtheorem{rem}[thm]{Remark}
 \newtheorem{eg}[thm]{Example}
 \numberwithin{equation}{section}

\begin{document}

\title[Hilbert basis of the Lipman semigroup]{Hilbert basis of the Lipman semigroup}


\author[Mesut \c{S}ah\.{i}n]{Mesut \c{S}ah\.{i}n}
\address{Department of Mathematics, \c{C}ank\i r\i ~Karatek\.{i}n University, 18100, \c{C}ank\i r\i, ~ Turkey}
\curraddr{} \email{mesutsahin@gmail.com}
\thanks{}



\date{\today}


\begin{abstract}
In this work, we give a new method to compute the Hilbert basis of the semigroup of certain positive divisors supported on the exceptional divisor of a normal surface singularity. Our approach is purely combinatorial which permits to avoid the long calculation of the invariants of the ring as it is presented in the work of Alt\i nok and Tosun.
\end{abstract}

\maketitle

\section{Introduction}

The exceptional divisor of a resolution of a singularity of a normal surface is a connected curve. The set of positive divisors supported on this exceptional divisor satisfying some negativity condition forms a semigroup, called the semigroup of Lipman in reference to his work \cite{lipman}. The unique smallest element of this semigroup characterizes the class of the singularity; for example, if the geometric genus of the smallest element is zero then the singularity is called rational \cite{artin}. When the singularity is rational, the elements of the semigroup of Lipman are in one-to-one correspondence with the functions in the local ring at the singularity. These elements are important to understand algebraic and topological structure of the
corresponding singularity, see \cite{tolga,nemethi1, nemethi2}.

The smallest element of the semigroup of Lipman is calculated by the Laufer algorithm (see \cite[4.1]{laufer}) and all the other elements are computed by the algorithms given in \cite{pinkham,tosun}.  The natural question of determining an explicit finite generating set for the semigroup is answered in \cite{meral-selma}. The authors use the tools from toric geometry to compute all the generators by means of the generators of certain ring of invariants. Their method is effective but it is difficult to follow for an exceptional divisor with many components.

Here we present an easier combinatorial method to obtain the set of generators of the semigroup of Lipman. More significantly, we describe another semigroup associated to an exceptional divisor whose Hilbert basis, which can be computed directly from the intersection matrix of the exceptional divisor, gives exactly the generators of the Lipman semigroup and the corresponding ring of invariants at the same time. The latter is important for a deeper study of properties of the associated toric variety, such as being a set-theoretic complete intersection \cite{bmt} or having a nice Castelnuovo-Mumford regularity \cite{regularity}.

\section{Preliminaries}
In this section, we recall some terminology and results which will be used later without any reference.
Let $Y$ be a normal surface with
an isolated singularity at $0$ and $(X,E)\rightarrow (Y,0)$ be a resolution of singularities with an exceptional curve $E$ over $0$. Let $E_1,\dots,E_n$ be the
irreducible components of $E$. The set of divisors
supported on $E$ forms a lattice defined by
$$M:=\{ m_1 E_1+\cdots+m_n E_n \,|\, m_i \in \Z \}.$$
There is an additive subsemigroup of $M$  which is referred to as the \textit{Lipman semigroup} and is defined by
$$\E:=\left\{ D\in M \,|\, D \cdot E_i \leq 0, \:\mbox{for any} \:i=1,\dots,n \right\}.$$

It follows that if $m_1 E_1+\cdots+m_n E_n \in \E \backslash \{0\}$ then $m_i >0$, for all $i=1,\dots,n$, see \cite{artin}. By definition, $D\in \E$ if and only if $D \cdot E_i=-d_i$ for some $d_i \in \N$ and for all $i=1,\dots,n$. Denote by $M(E)$ the intersection matrix of the exceptional divisor $E$, that is, a matrix with integral entries defined by the intersection multiplicities $E_i \cdot E_j$. It is known that $M(E)$ is negative definite.

Given $\dis D=m_1 E_1+\cdots+m_n E_n \in M$, with $m_i > 0$. The following equivalence determines the elements of $\E$
\begin{equation}\label{eq1} \dis D \cdot E_i=-d_i \Leftrightarrow M(E) [\, m_1 \; \cdots \; m_n \,]^T=-[\, d_1 \; \cdots \; d_n \,]^T.
 \end{equation}

If $\e_i=[\, 0 \; \cdots \;1 \; \cdots \; 0 \,]^T$ is the standard basis element of the space of column matrices of size $n$, then every column matrix $-[\, d_1 \; \cdots \; d_n \,]^T$, with all $d_i \geq 0$, is spanned by $-\e_1,\dots,-\e_n$. Hence, it follows that the rational cone over $\E$ is generated by the $F_i$ which is defined to be the (rational) solution of the matrix equation above corresponding to $-\e_i$ for each $i$. Therefore, we can write $F_i$ as follows:
$$F_i=\sum_{j=1}^{n} \frac{a_{ij}}{b_{ij}} E_j,$$ where $a_{ij}$ and $b_{ij}$ are relatively prime integers. Now, let $g_i$ be the least common factor of $b_{i1},\dots,b_{in}$ so that $g_iF_i$ is the smallest multiple of $F_i$ that belongs to $\E$. Denote by $M'$ the lattice generated by $F_1,\dots,F_n$ and let $N$, $N'$ be the corresponding dual lattices of $M$, $M'$ respectively. Then, $N'$ is a sublattice of $N$ of finite index, since $M$ is a sublattice of $M'$.

Denote by $\check{\sigma}$ the cone in $M_{\R}:=M\otimes_{\Z}\R$ spanned by the semigroup $\E$. The semigroup $\check{\sigma} \cap M \supseteq \E$ is called the saturation of $\E$ and the semigroup $\E$ itself is called saturated (or normal) if $\check{\sigma} \cap M \subseteq \E$ as well.

\begin{prop} $\E$ is a pointed, saturated semigroup which is also simplicial and finitely generated.
\end{prop}
\begin{proof} If $D \in \E$, then $D \cdot E_i \leq 0$ which forces that $-D \cdot E_i \geq 0$. This means that $D\in \E \cap (-\E)$ if and only if $D=0$, which proves that $\E$ is pointed.

Now, take $D\in \check{\sigma} \cap M$, i.e. $D=mD'$, for some $D'\in \E$ and $m>0$. Since $D'\in \E$, we have $D' \cdot E_i \leq 0$ which yields immediately that $D \cdot E_i=m D' \cdot E_i \leq 0 $. Therefore, $D$ must belong to $\E$ which reveals that $\E$ is saturated.

Since $\E$ is saturated it follows that $\E=\check{\sigma} \cap M$ and thus $\check{\sigma}$ is generated by $n=\dim \check{\sigma}=\rank M$ linearly independent elements $F_1,\dots,F_n$ over $\Q^+$, which means that $\check{\sigma}$ is a \textit{maximal} and \textit{simplicial} strongly convex rational polyhedral cone. This shows that $\E$ is simplicial.

That $\E$ has a unique finite minimal generating set $\H_{\E}$ over $\N$ follows directly from \cite[Lemma 13.1]{sturmfels}.
\end{proof}

\begin{defn} The unique minimal generating set $\H_{\E}$ of $\E$ over $\N$ is called the \textit{Hilbert basis} of $\E$.
\end{defn}
Since $\E$ is saturated, we can associate a normal toric variety $V_{\E}:=\spec \C[\E]$ to $\E$, see \cite{fulton} for details. It turns out that the coordinate ring $\C[\E]$ of this variety is nothing but the ring of invariants of $\C[M']$ under the natural action of $N/N'$, see \cite[Proposition 3.4]{meral-selma}.

\begin{rem} $V_{\E}$ is isomorphic to the geometric quotient $\C^k / G$ in the language of the Geometric Invariant Theory, since $G=N/N'$ is a finite group and $\E$ is simplicial. Hence, $V_{\E}$ has only quotient singularities.
\end{rem}

 \section{Main Results}
Recall that the unique minimal generating set $\H_{S}$ of a pointed, saturated semigroup $S$ is called the \textit{Hilbert basis} of $S$, see \cite{sturmfels}. We first associate to $\E$ the obvious subsemigroup of $\N^n$;
 $$\dis S_1:=\{(m_1,\dots,m_n) \in \N^n \,|\,  m_1 E_1+\cdots+m_n E_n \in \E \}.$$

\begin{prop}\label{S1} $\phi_1: \E \rightarrow S_1$ is an isomorphism, where $\phi_1(m_1 E_1+\cdots+m_n E_n)=(m_1,\dots,m_n)$. \hfill $\Box$
\end{prop}

Similarly, we can associate another subsemigroup $S_2$ of $\N^n$ with the semigroup $\E$ as follows:
 $$\dis S_2:=\{(d_1,\dots,d_n) \in \N^n \,|\, d_i=-(D\cdot E_i), \:\mbox{for some} \: D\in \E \;\:\mbox{and for all} \: i=1,\dots,n\}.$$

\begin{prop}\label{S2} $S_2$ and $\E$ are isomorphic as semigroups. Moreover, the Hilbert basis of $S_2$ determines the parametrization of the toric variety $V_{\E}$.
\end{prop}
\begin{proof} Define $\phi_2: \E \rightarrow S_2$ by $\phi_2(D)=(-D\cdot E_1,\dots,-D\cdot E_n)$, for each $D \in \E$. This defines clearly a homomorphism between the semigroups, since we have $$(D+D')\cdot E_i=D\cdot E_i+D'\cdot E_i, \quad \mbox{for any}\quad i=1,\dots,n.$$

Surjectivity follows from the Equation $1$ together with $M(E)$ being invertible over the rationals. Indeed, for a given $(d_1,\dots,d_n)\in S_2$ there are non-negative rational numbers $m'_i$ such that $\dis [\, m'_1 \; \cdots \; m'_n \,]^T=-(M(E))^{-1}[\, d_1 \; \cdots \; d_n \,]^T$. Multiplying $m'_i$ by the least common factor of the positive integers in the denominators of $m'_i$, we get non-negative integers $m_i$ such that $\phi_2(D)=(d_1,\dots,d_n)$, where $D=m_1 E_1+\cdots+m_n E_n \in \E$. The injectivity follows similarly.

We prove the second part now. Since $\C[\E]$ and $\C[S_2]$ are isomorphic from the first part, $V_{S_2}$ is an embedding of $V_{\E}=\spec \C[\E]$ in some affine space. It is known that $\C[S_2]$ is generated minimally by the monomials $u_1^{d_1}\cdots u_n^{d_n}$ which is determined by $(d_1,\dots,d_n) \in \H_{S_2}$. Therefore we need to determine the elements of $\H_{S_2}$ more precisely. Since $S_2$ is a subsemigroup of $\N^n$ and $\phi_2(g_iF_i)=g_i \e_i$ is the smallest element of $S_2$ on the $i$-th ray of the cone $\phi_2(\check{\sigma})$, it follows that $H_{S_2}$ contains $g_i \e_i$, for each $i=1,\dots,n$. If we denote by $\h_j=h_{j1}\e_1+\cdots+h_{jn}\e_n$ the other elements of the Hilbert basis of $S_2$, then it follows from \cite[Corollary 2]{katsabekis-thoma} that the toric variety $V_{S_2}$ is parametrized by the toric set
$$\Gamma(S_2)=\{(u_1^{g_1},\dots,u_n^{g_n},u_1^{h_{11}}\cdots u_n^{h_{1n}},\dots,u_1^{h_{k1}}\cdots u_n^{h_{kn}})\;|\;u_1,\dots,u_n \in \C\}.  $$
\end{proof}

In order to state our main result, let $A=[M(E)|I_n]$ be the $n \times 2n$ integer matrix obtained by joining the intersection matrix $M(E)$ of the exceptional divisor $E$ and the identity matrix of size $n \times n$. Then, we define the last semigroup as $$\dis S=\{(v_1,\dots,v_{2n}) \in \N^{2n} \,\big|\, A\cdot[v_1\cdots v_{2n}]^T=0\}.$$
Here is the nice relation between the three semigroups defined so far.
\begin{thm}\label{main} $\dis S = S_1 \times_{\mathcal{E}} S_2 := \{ (\x,\y) \in S_1 \times S_2 \mid \phi_1^{-1}(\x) = \phi_2^{-1}(\y) \}$
where $\phi_i^{-1} : S_i \to \mathcal{E}$ are the isomorphisms introduced in Propositions $2$ and $3$.
\end{thm}
\begin{proof} The following observations can be seen immediately.
\begin{eqnarray*} (v_1, \dots, v_{2n})\in S &\Leftrightarrow& A \cdot [v_1\cdots v_{2n}]^T=0 \Leftrightarrow M(E)\cdot [v_1\cdots v_{n}]^T=-[v_{n+1}\cdots v_{2n}]^T\\
&\Leftrightarrow& D=v_1E_1+\cdots+v_nE_n \in \E \quad \mbox{and}\quad D\cdot E_i=-v_{n+i}, \;\mbox{for any} \;i=1,\dots,n\\
&\Leftrightarrow& (v_1,\dots, v_{n},v_{n+1},\dots, v_{2n}) \in S_1 \times_{\mathcal{E}} S_2.
\end{eqnarray*}
Therefore, the proof is complete.
\end{proof}
The Hilbert basis of this last semigroup is easy to find and gives important information about the others as we see now.

\begin{corollary} Hilbert basis of $S$ gives the generators of the Lipman semigroup and the parametrization of the corresponding toric variety at the same time.
\end{corollary}
\begin{proof} By Theorem \ref{main}, it follows that the elements of $\H_{S}$ is in bijection with the elements of $\H_{S_1}$ and $\H_{S_2}$. Hence, $(m_1,\dots,m_n,d_1,\dots,d_n) \in \H_{S}$ if and only if $(m_1,\dots,m_n)\in \H_{S_1}$ and $(d_1,\dots,d_n) \in \H_{S_2}$. Now, it is clear from Proposition \ref{S1} that $(m_1,\dots,m_n)\in \H_{S_1}$ if and only if $m_1E_1+\cdots+m_nE_n \in \H_{\E}$. On the other hand, we know from the proof of Proposition \ref{S2} that $\H_{S_2}$ determines the parametrization of the toric variety associated to $\E$.
\end{proof}

\begin{rem} Our main Theorem \ref{main} gives rise to an algorithm which starts with the intersection matrix $M(E)$ and computes the Hilbert basis $\H_{\E}$ of the Lipman semigroup and the parametrization of the toric variety $V_\E$ at once. It uses existing algorithms for computing Hilbert basis of lattice points of cones, where the lattice is given by the kernel of an integral matrix $A$, see \cite{hemmecke} and references therein or \cite[Chapter 6]{kr}.
\end{rem}

We conclude the paper with an illustration of our user-friendly combinatorial method.

\begin{eg} Consider the exceptional divisor $E$ over a singularity of $A_2$-type. Then $A=\left[
        \begin{array}{cccc}
          -2 & \;\;\;1 &\;\;1&\;\;0 \\
          \;\;1 &-2 &\;\;0 &\;\;1\\
        \end{array}
      \right].$

A computation with a computer package (e.g. CoCoA \cite{cocoa} or 4ti2 \cite{4ti2}) gives the Hilbert basis of $S$ to be the set
      $$\H_S=\{(2,1,3,0),(1,1,1,1),(1,2,0,3)\}.$$
This says that $\H_\E=\{2E_1+E_2,E_1+E_2,E_1+2E_2\}$ and the smallest element $E_1+E_2$ is the fundamental cycle of $\E$. Since $\H_{S_2}=\{(3,0),(1,1),(0,3)\}$, it also says that the corresponding toric variety $V_\E$ is parametrized by the toric set $\Gamma(S_2)=\{(u_1^3,u_1u_2,u_2^3)\;|\;u_1,u_2 \in \C\}$.
\end{eg}

\section*{Acknowledgment}
The paper has been written while the author was visiting the Abdus Salam International Centre for Theoretical Physics (ICTP), Trieste, Italy. The author thanks the Department of Mathematics of ICTP and \c{C}ank\i r\i ~Karatek\.{i}n University for their support. He would like to thank M. Tosun for stimulating discussions and
her valuable comments on the article. He also thanks the referee for his/her careful reading.

\end{document}